\documentclass[12pt]{amsart}
\usepackage{geometry}
\usepackage{amssymb}
\usepackage{stmaryrd}
\usepackage{graphicx}
\usepackage[colorlinks=true, citecolor=blue, linkcolor=blue]{hyperref}

\usepackage{ifdraft}
\ifdraft{
\newgeometry{asymmetric}
\usepackage[notref, notcite]{showkeys}
\usepackage[bordercolor=white, color=white, colorinlistoftodos]{todonotes}
}{
\usepackage[disable]{todonotes}
}

\makeatletter\providecommand\@dotsep{5}\def\listtodoname{List of Todos}\def\listoftodos{\hypersetup{linkcolor=black}\@starttoc{tdo}\listtodoname\hypersetup{linkcolor=blue}}\makeatother
\newtheorem{lemma}{Lemma}

\newtheorem{theorem}{Theorem}

\newtheorem{corollary}{Corollary}
\newtheorem{remark}{Remark}

\def\R{\mathbb R}
\def\Q{\mathbb Q}

\def\p{\partial}

\newcommand{\pair}[1]{\left\langle #1 \right\rangle}
\newcommand{\norm}[1]{\Vert#1\Vert}

\newcommand{\jump}[1]{\llbracket#1\rrbracket}
\newcommand{\tnorm}[1]{\vert\hspace{-0.3mm}\Vert#1\Vert\hspace{-0.3mm}\vert}
\def\H{\mathcal H}

\date{Dec 2015, last edit by Lauri, compiled \today}

\title[Data assimilation for the heat equation]{Data assimilation for the heat equation using stabilized finite element methods}
\author{Erik Burman}
\address{Department of Mathematics, University College London, Gower Street, London UK, WC1E 6BT.}
\email{e.burman@ucl.ac.uk}
\author{Lauri Oksanen}
\address{Department of Mathematics, University College London, Gower Street, London UK, WC1E 6BT.}
\email{l.oksanen@ucl.ac.uk}
\begin{document}
\begin{abstract}
We consider data assimilation for the heat equation using a finite
element space semi-discretization. The approach is optimization
based, but the design of regularization operators and parameters rely
on techniques from the theory of stabilized finite elements. The
space semi-discretized system is shown to admit a unique
solution. Combining sharp estimates of the numerical stability of the
discrete scheme and conditional stability estimates of the
ill-posed continuous pde-model we then derive error estimates that
reflect the approximation order of the finite element space and the
stability of the continuous model. Two different data assimilation situations with different stability properties are considered to illustrate the framework. Full detail on how to adapt known stability estimates for the continuous model to work with the numerical analysis framework is given in appendix.
\end{abstract}
\maketitle
\section{Introduction}
We consider two data assimilation problems for the heat equation 
\begin{align}
\label{eq_heat}
\p_t u - \Delta u = f, \quad \text{in $(0,T) \times \Omega$},
\end{align}
where $T > 0$ and $\Omega \subset \R^n$ is open. 
Let $\omega, B \subset \Omega$ be open
and let $0 < T_1 < T_2 \le T$.
Both the data assimilation problems are of the following general form: 
determine the restriction $u|_{(T_1, T_2) \times B}$ of a solution to the heat equation (\ref{eq_heat})
given $f$ and the restriction
$u|_{(0,T) \times \omega}$.
The crucial difference between the two problems is that in the first problem we do not make any assumptions on the boundary condition satisfied by $u$, whereas in 
the second problem we assume that it satisfies the lateral Dirichlet boundary condition $u|_{(0,T) \times \p \Omega} = 0$. We emphasize that, in both the problems, no information on the initial condition satisfied by $u$ is assumed to be known. 

For the first problem we assume the following geometry,
\begin{itemize}
\item[(A1)] $B$ is connected, $\omega \subset B$, and
the closure $\overline B$ is compact and contained in $\Omega$,
\end{itemize}
and in the second case we may choose $B = \Omega$ while assuming that
\begin{itemize}
\item[(A2)] $\overline \Omega$ is a compact, convex, polyhedral domain.
\end{itemize}

It is known that the first problem is conditionally H\"older stable
and that this stability is optimal. More precisely, the following continuum stability estimate holds.

\begin{theorem}
\label{th_cont_unstable}
Let $\omega \subset \Omega$ be open and non-empty, and let $0 < T_1 < T_2 < T$. 
Suppose that an open set $B \subset \Omega$ satisfies (A1).
Then there are $C > 0$ and $\kappa \in (0,1)$
such that for all $u$ in the space
\begin{align}
\label{energy_space_nobc}
H^1(0,T; H^{-1}(\Omega)) \cap L^2(0,T; H^1(\Omega)),
\end{align}
it holds that 
\begin{align*}
&\norm{u}_{L^2(T_1, T_2; H^1(B))} 
\le C (\norm{u}_{L^2((0, T) \times \omega)} + 
\norm{L u}_{(0,-1)})^\kappa (\norm{u}_{L^2((0, T) \times \Omega)} + 
\norm{L u}_{(0,-1)})^{1-\kappa},
\end{align*}
where $L = \p_t - \Delta$ and $\norm{\cdot}_{(0,-1)} = \norm{\cdot}_{L^2(0, T; H^{-1}(\Omega))}$.
\end{theorem}

Here the factor with the power $\kappa$ can be viewed as the size of the data 
\begin{align}
\label{the_data}
q = u|_{(0, T) \times \omega}, \quad f = Lu, 
\end{align}
and the norm $\norm{u}_{L^2((0, T) \times \Omega)}$
as an apriori bound for the unknown $u$.
Let us also emphasize that the assumption $\overline B \subset \Omega$ is essential, indeed, if $\overline B \cap \p \Omega \ne \emptyset$ then the optimal stability is conditionally logarithmic.
We outline the continuum stability theory in the appendix below. The second problem is stable.

\begin{theorem}
\label{th_cont_stable}
Let $\omega \subset \Omega$ be open and non-empty, and let $0 < T_1 < T$. 
Suppose that (A2) holds.
Then there is $C > 0$ such that for all $u$ in the space
\begin{align}
\label{energy_space}
H^1(0,T; H^{-1}(\Omega)) \cap L^2(0,T; H_0^1(\Omega)),
\end{align}
it holds that 
\begin{align*}
&\norm{u}_{C(T_1, T; L^2(\Omega))} +
\norm{u}_{L^2(T_1, T; H^1(\Omega))} +
\norm{u}_{H^1(T_1, T; H^{-1}(\Omega))} 
\\&\quad\le C (\norm{u}_{L^2((0, T) \times \omega)} + 
\norm{L u}_{(0,-1)}).
\end{align*}
%\begin{align*}
%&\norm{u}_{L^2(T_1, T; H^1(\Omega))} 
%\le C (\norm{u}_{L^2((0, T) \times \omega)} + 
%\norm{L u}_{0,-1}).
%\end{align*}
\end{theorem}

%Here the assumption $T_1 > 0$ is essential, otherwise the optimal stability is conditionally logarithmic.
%\HOX{Removed here a small remark on the case $T_1 = 0$ as I could not find a good reference.}

In what follows, we call the two data assimilation problems the unstable and the stable problem, respectively.
We use the shorthand notation 
\begin{align*}
a(u,z) = (\nabla u, \nabla z), 
\quad
G_f(u,z) = (\p_t u, z) + a(u,z) - \pair{f,z},
\quad 
G = G_0,
\end{align*}
where $(\cdot,\cdot)$ is the inner product of $L^2((0,T) \times \Omega)$ and
$\pair{\cdot,\cdot}$ is the dual pairing between 
$L^2(0,T; H^{-1}(\Omega))$ and $L^2(0,T; H_0^1(\Omega))$.
Note that for $u \in H^1((0,T) \times \Omega)$, the equations
\begin{align}
\label{weak_lapl_eq}
G_f(u,z) = 0, \quad z \in L^2(0,T; H_0^1(\Omega)),
\end{align}
give the weak formulation of $\p_t u- \Delta u = f$.
Our approach to solve the data assimilation problem, in both the cases, is based on minimizing the Lagrangian functional
\begin{align}
\label{Lagrangian}
\L_{q,f}(u,z) = \frac12 \|u - q\|_{\omega}^2 +
 \frac12 s(u,u) - \frac12 s^*(z,z),
+ G_f(u,z),
\end{align}
where the data $q$  
and $f$ are fixed.
Here 
$\norm{\cdot}_\omega$ is the norm of $L^2((0,T) \times \omega)$,
and $s$ and $s^*$ are symmetric bilinear forms, that are chosen differently in the two cases, and that we call the primal and dual stabilizers, respectively.
Note that minimizing $\L_{q,f}$ can be seen as fitting 
$u|_{(0,T) \times \omega}$ to the data $q$ under the constraint (\ref{weak_lapl_eq}), 
$z$ can be interpreted as a Lagrange multiplier,
and $s/2$ and $s^*/2$ as regularizing penalty terms.

In this paper we consider only semi-discretizations, that is,
we minimize $\L_{q,f}$ on a scale of spaces that are discrete in the spatial variable but not in the time variable. 
The spatial mesh size $h > 0$
will be the only parameter controlling the convergence of the
approximation, and we use piecewise affine finite elements. 

%\HOX{I edited this a bit. We can revert back to the prev. version if you don't agree with the changes}
An important feature of the present work is that the choice of the
regularizing terms is driven by the analysis and designed to give
error estimates %for clean data 
that reflect the approximation
properties of the finite element space and the stability of the continuous model. 
In particular, when the continuous model is Lipschitz stable, %for energy norm 
we obtain optimal error estimates in the usual sense in
numerical analysis, that is, the estimates are optimal compared to interpolation of the exact solution. This is the case for the
second model problem introduced above.
The regularization is constructed on the discrete level,
that is, $s$ and $s^*$ are defined on the semi-discrete spaces, and in some cases, they may not even make sense on the continuous level.

We show how the different stability properties of the
two model problems lead to different regularization
operators. The choice for a given problem is not unique, and 
the design of regularizing terms leading to optimal estimates can also be driven by computational considerations, such as computational cost or
couplings in the discrete formulations.
%The approach presented here is general and can be applied to a large class of problems and finite element methods, 
Therefore we present an abstract framework for the design of regularization operators. 
When measurement errors are present we also show how to quantify the damping of
perturbations.

For the unstable problem,
under suitable choices of the semi-discrete spaces
and the regularization, we show that (\ref{Lagrangian})
has a unique minimizer $(u_h, z_h)$, $h > 0$, 
satisfying the following convergence rate
\begin{align}
\label{convrate_informal}
\norm{u_h - u}_{L^2(T_1, T_2; H^1(B))} 
&\le C h^\kappa (\|u\|_{*} + \|f\|),
\end{align}
where $u$ is the unique solution of the continuum data assimilation problem, $\kappa \in (0,1)$ is the constant in Theorem \ref{th_cont_unstable}, and the norms on the right-hand side capture the assumed apriori smoothness, see Theorem \ref{th_main_unstable} below for the precise formulation. 
For the stable problem, we establish an analogous result but with $\kappa = 1$, that is, the convergence rate reflects now the stability estimate in Theorem \ref{th_cont_stable}.
Moreover, in the stable case, we can replace the norm on the left-hand side of (\ref{convrate_informal}) with the stronger norm in Theorem \ref{th_cont_stable}, see Theorem \ref{th_main_stable_abs} below for the precise formulation. 

\subsection{Previous literature}
The classical approach to data assimilation and inverse problems is to
introduce Tikhonov regularization on the continuous level \cite{TA77,Bau87,EHN96}
and then discretize the well-posed regularized model. The use of
Carleman estimates for uniqueness of inverse problems and the
quantification of stability was first proposed in \cite{BK81}. The use of conditional
stability to choose a Tikhonov regularization parameter has been explored in
\cite{CY00}. Nevertheless a common feature of all methods where
regularization takes place on the continuous level is that the
accuracy of the reconstruction will be determined by the
regularization error, leaving no room to exploit the interplay
between discretization and regularization. 

The method we propose is an instance of the so-called 4DVAR method
\cite{DT86}. The analysis in the present paper, which draws on
previous ideas in the stationary case \cite{Bu13,Bu14, BHL16}, is to
our best knowledge the first complete numerical analysis of a 4DVAR
type method. The main characteristic of our approach is to separate
the numerical stability and the stability of the continuous model, and
use the continuous stability estimate in the perturbation analysis.

In contrast, another recent approach to parabolic data assimilation
problems is to derive Carleman estimates directly on the discrete
scheme \cite{BHR10,BHR11}, which may then be used for convergence
analysis. Other methods for data assimilation include the so called
back and forth nudging \cite{Auroux2005} and methods using
null-controllability \cite{Puel2009, GOP11}. In the former case,
forward and backward solves are combined with filtering techniques (or
penalty) to drive the approximate solution close to data, and the
latter case leads to an approximation algorithm which uses auxiliary
optimal control problems. For a discussion of different approaches to
numerical approximation of control and inverse problems we
refer to the monograph \cite [Chapter 5]{Ervedoza2013}.

\section{Spatial discretization}

We begin by observing that we may assume without loss of generality that $\Omega \subset \R^n$ is a compact and connected polyhedral domain also in the case (A1). Indeed, we can choose a compact and connected polygonal domain $\tilde \Omega \subset \Omega$ such that $\overline B \subset \tilde \Omega$
and then replace $\Omega$ by $\tilde \Omega$.

Consider a family $\mathcal T = \{\mathcal{T}_h;\ h > 0\}$ of triangulations of $\Omega$ consisting of simplices %$\{K\}$
such that the intersection of any two distinct simplices is
either a common vertex, a common edge or a common face.
Moreover, assume that the family $\mathcal{T}$ is quasi uniform,
see e.g. \cite[Def. 1.140]{Ern2004}, and
indexed by 
\def\diam{\mbox{diam}}
$$
h = \max_{K \in \mathcal{T}_h} \diam(K).
$$ 
The family $\mathcal{T}$
satisfies the following trace inequality, see e.g. \cite[Eq. 10.3.9]{BS08},
\begin{equation}\label{trace_cont}
\|u\|_{L^2(\partial K)} \leq C (h^{-\frac12} \|u\|_{L^2(K)} + h^{\frac12}
\|\nabla u\|_{L^2(K)}), \quad u \in H^1(K),
\end{equation}
where the constant $C$ is independent of $K \in \mathcal T_h$ and $h>0$.

Let $V_h$ be the $H^1$-conformal approximation space based on the $\mathbb{P}_1$ finite element, that is, 
\begin{align}
\label{def_Vh}
V_h = \{u \in C(\bar \Omega):\ u\vert_K \in \mathbb{P}_1(K),\ 
%\quad\forall 
K \in \mathcal{T}_h\},
\end{align}
where $\mathbb{P}_1(K)$ denotes the set of polynomials of degree less
than or equal to $1$ on $K$. 
We recall that the family $\mathcal{T}$
satisfies the following discrete inverse inequality,
see e.g. \cite[Lem. 1.138]{Ern2004},
\begin{equation}
\label{inverse_disc}
\|\nabla u\|_{L^2(K)} \leq C
h^{-1} \|u\|_{L^2(K)} , \quad u \in \mathbb{P}_1(K),
\end{equation}
where the constant $C$ is independent of $K \in \mathcal T_h$ and $h>0$.

We choose an interpolator
$$
\pi_h : H^1(\Omega) \to V_h, \quad h > 0,
$$
that satisfies the following stability and approximation properties
\begin{align}
\label{pi_h_stab}
\norm{\pi_h u}_{H^1(\Omega)}
&\le C \norm{u}_{H^{1}(\Omega)}, &u\in H^{1}(\Omega),\ h > 0,\
\\\label{pi_h}
\norm{u - \pi_h u}_{H^m(\Omega)}
&\le C h^{k-m}\norm{u}_{H^{k}(\Omega)}, &u\in H^{k}(\Omega),\ h > 0,\
\end{align}
where $k=1,2$ and $m = 0, k-1$, 
and that preserves vanishing Dirichlet boundary conditions, that is, 
using the notation $W_h = V_h \cap H_0^1(\Omega)$,
\begin{align}
\label{pi_bc}
\pi_h : H_0^1(\Omega) \to W_h, \quad h > 0.
\end{align}
A possible choice is the Scott-Zhang interpolator \cite{Scott1990}.

\subsection{Spatial jump stabilizer}

In the unstable case, the regularization will be based on the spatial jump stabilizer that we will introduce next.
We denote by
$\mathcal{F}_h$ the set of internal faces of $\mathcal{T}_h$,
and define for $F \in \mathcal{F}_h$,
\begin{align*}
%\jump{u}_F &= u|_{K_1} n_1 + u|_{K_2} n_2, \quad u \in H^1(\Omega),
%\\
\jump{n \cdot u}_F &= n_1 \cdot u|_{K_1}  + n_2 \cdot u|_{K_2} , \quad u \in H^1(\Omega; \R^d),
\end{align*}
where $K_1, K_2 \in \mathcal{T}_h$ are the two simplices satisfying 
$K_1 \cap K_2 = F$, and $n_j$ is the outward unit normal vector of $K_j$, $j=1,2$.
We define the jump stabilizer
\def\J{\mathcal J}
\begin{align}
\label{def_s}
\J(u,u) = \sum_{F \in \mathcal{F}_h} \int_F h \jump{n \cdot
  \nabla u}_F^2 ~\mbox{d}s,
\quad u \in V_h,
\end{align}
where $\mbox{d}s$ is the Euclidean surface measure on $F$.

\begin{lemma}[Discrete Poincar\'e inequality]
%Suppose that $\Omega$ is convex.
There is $C > 0$ such that all 
$u \in V_h$, 
and $h > 0$ satisfy
\begin{align}
\label{J_poincare}
\norm{u}_{L^2(\Omega)} \le C h^{-1} (\J(u,u)^{1/2} + \norm{u}_{L^2(\omega)}).
\end{align}
\end{lemma}
\begin{proof}
The original results on Poincar\'e inequalities for piecewise
$H^1$-functions may be found in \cite{Brenner03}. For a detailed proof
of \eqref{J_poincare} see \cite{BHL16}.
\end{proof}

\begin{lemma}\label{form_prop_lem}
There is $C > 0$ such that all 
$u \in V_h$, 
$v \in H_0^1(\Omega)$, 
$w \in H^2(\Omega)$
and $h > 0$ satisfy
\begin{align}
\label{J_lb}
(\nabla u, \nabla v)_{L^2(\Omega)} &\le 
C \J(u,u)^{1/2} \left( h^{-1} \|v\|_{L^2(\Omega)} + \|v\|_{H^1(\Omega)} \right),
\\\label{J_pi}
\J(\pi_h w, \pi_h w) &\le C h^2 \|w\|_{H^2(\Omega)}^2.
\end{align}
\end{lemma}
\begin{proof}
Towards (\ref{J_lb}) we integrate by parts and recall that $u$ is an affine function on each element to obtain
\begin{align}
\label{J_lb_eq}
(\nabla u, \nabla v)_{L^2(\Omega)} =
\sum_{K \in \mathcal T_h} \int_K \nabla u \cdot \nabla v ~\mbox{d}x 
= h^{-1/2} \sum_{F \in \mathcal F_h} \int_F h^{1/2} \jump{n \cdot
  \nabla u}_F\, v ~\mbox{d}s.
\end{align}
Thus by the Cauchy-Schwarz inequality
$$
(\nabla u, \nabla v)_{L^2(\Omega)}
\le h^{-1/2} s(u,u)^{1/2} 
\left(\sum_{F \in \mathcal{F}_h} \int_F v^2 \mbox{d}s \right)^{1/2},
$$
and the inequality (\ref{J_lb}) follows by applying (\ref{trace_cont}) to the last factor.

Let us now turn to (\ref{J_pi}). As
$\jump{n \cdot \nabla w}_F = 0$, 
we have using (\ref{trace_cont})
\begin{align*}
&\int_F h \jump{n \cdot \nabla \pi_h w}_F^2 ~\mbox{d}s
=h \int_F \jump{n \cdot \nabla (\pi_h w - w)}_F^2 ~\mbox{d}s
\\
&\quad\le C \sum_{j=1,2} (
\norm{n_j \cdot \nabla (\pi_h w - w)}_{L^2(K_j)}^2 
%\\&\qquad\qquad
+ h^2 \norm{n_j \cdot \nabla (\pi_h w - w)}_{H^1(K_j)}^2).
\end{align*}
Summing over all internal faces and observing that $$\|\nabla (\pi_h
w - w)\|_{H^1(K_j)}^2 = \|\nabla (\pi_h
w - w)\|_{L^2(K_j)}^2+ \norm{\nabla w}_{H^1(K_j)}^2,$$ 
we obtain
\begin{align*}
&\sum_{F \in \mathcal{F}_h} \int_F h \jump{n \cdot \nabla \pi_h w}_F^2 ~\mbox{d}s \leq C \left(
  \norm{\nabla (\pi_h w - w)}_{L^2(\Omega)}^2 +  h^2 
  \norm{w}^2_{H^2(\Omega)}\right).
\end{align*}
The first term on the right-hand side satisfies
$$
\norm{\nabla (\pi_h w - w)}_{L^2(\Omega)}^2 
\le C h^2 \norm{w}_{H^2(\Omega)}^2,
$$
by (\ref{pi_h}).
\end{proof}

\section{The unstable problem}

There are several possible choices for the primal and dual stabilizers
$s$ and $s^*$ in (\ref{Lagrangian}), and different choices lead to numerical methods that may differ in terms of practical performance. 
To illustrate the main ideas of our approach, we begin by considering a concrete choice of the stabilizers in Section \ref{sec_model_case}, and give an abstract framework in Section \ref{sec_framework}.

In what follows, we use the shorthand notations $H^{(k,m)} = H^k(0,T; H^m(\Omega))$, 
$$
\norm{u}_{(k,m)} = \norm{u}_{H^k(0,T; H^m(\Omega))}, \quad k, m \in \R,
$$
and $H_0^{(k,m)} = H^{(k,m)} \cap L^2(0,T; H_0^1(\Omega))$. 
We recall also that $\norm{u} = \norm{u}_{(0,0)}$
and that $\norm{u}_\omega$ is the norm of $L^2((0,T) \times \omega)$.

\subsection{A model case}
\label{sec_model_case}

We choose the following primal and dual stabilizers
\begin{align}
\label{stab_simple}
s(u,u) = \int_0^T \J(u,u) dt + \norm{h \p_t u}^2,
\quad s^* = a,
\end{align}
defined on the respective semi-discrete spaces
\def\V{\mathcal V}
\def\W{\mathcal W}
\begin{align}
\label{stab_spaces_simple}
\V_h = H^1(0,T; V_h), \quad \W_h = L^2(0,T; W_h).
\end{align}
We define also the semi-norm and norm
\begin{align}
\label{stab_norm_simple}
|u|_\V = s(u,u)^{1/2}, \quad 
\norm{z}_\W = s^*(z,z)^{1/2}, \quad u \in \V_h,\ z \in \W_h,\ h > 0.
\end{align}
Recall that $W_h = V_h \cap H_0^1(\Omega)$ and whence the Poincar\'e inequality
\begin{align}
\label{ds_poincare}
\|z\|_{L^2(\Omega)} \le C \|\nabla z\|_{L^2(\Omega)}, \quad z \in H_0^{1}(\Omega),
\end{align}
implies that $\norm{\cdot}_\W$ is indeed a norm on $\W_h$.

\begin{lemma}
The semi-norm $|\cdot|_\V$ and norm $\norm{\cdot}_\W$
satisfy the following inequalities with a constant $C > 0$ that is independent from $h > 0$: lower bounds
\begin{align}
\label{ps_lower}
G(u,z) 
&\le C |u|_\V \left(h^{-1} \|z\| + \|z\|_{(0,1)}\right),
& u \in \V_h,\ z \in H_0^{(0,1)},
\\\label{ds_lower}
G(u,z)  
&\le C \norm{z}_{\W}\left(\|u\|_{(1,0)} + \|u\|_{(0,1)}\right),
& u \in H^{(1,0)} \cap H^{(0,1)},\ z \in \W_h,
\end{align}
upper bounds
\begin{align}
\label{ps_upper}
|\pi_h u|_\V &\le C h (\norm{u}_{(1,1)} + \norm{u}_{(0,2)}), &u \in H^{(1,1)} \cap H^{(0,2)},
\\
\label{ds_upper}
\norm{\pi_h z}_\W &\le C \norm{z}_{(0,1)}, &z \in H^{(0,1)}_0,
\end{align}
and finally, for the semi-norm $|\cdot|_\V$ only, the Poincar\'e type inequality 
\begin{align}
\label{ps_poincare}
\norm{u} \le C h^{-1} (|u|_\V + \norm{u}_\omega), \quad u \in \V_h.
\end{align}
\end{lemma}
\begin{proof}
We have 
$$
G(u,z) = (\p_t u, z) + a(u,z) \le \norm{h \p_t u} h^{-1} \norm{z} + C \int_0^T \J(u,u) dt \,(h^{-1} \norm{z} + \norm{z}_1),
$$
where the bound for $a(u,z)$ follows from (\ref{J_lb_eq}), analogously to (\ref{J_lb}), after integration in time. This implies that (\ref{ps_lower}) holds.

The lower bound (\ref{ds_lower}) follows from the Cauchy-Schwarz inequality
$$
(\p_t u, z) + a(u,z)
\le \norm{\p_t u} \norm{z} + a(u,u)^{1/2} a(z,z)^{1/2},
$$
together with the Poincar\'e inequality (\ref{ds_poincare}).

Towards the upper bound (\ref{ps_upper}), we have
$$
|\pi_h u|_\V^2 = \int_0^T \J(\pi_h u,\pi_h u) + h^2 \norm{\p_t \pi_h u}^2
\le C h^2 \norm{u}_{(0,2)}^2 + C h^2 \norm{u}_{(1,1)}^2,
$$
where the bound for the first term 
follows from (\ref{J_pi}). The bound for the second term holds
by the $H^1$-stability (\ref{pi_h_stab}) of the interpolator $\pi_h$. 

The upper bound (\ref{ds_upper}) follows immediately from the $H^1$-stability of $\pi_h$, and
the lower bound (\ref{ps_poincare}) follows from (\ref{J_poincare}).
\end{proof}

Observe that, by the lower bound (\ref{ps_poincare}),
\[
\tnorm{(u,z)} = |u|_\V + \norm{u}_{\omega} + \norm{z}_{\W},
\]
is a norm on $\V_h \times \W_h$.

Let $q \in L^2((0,T) \times \omega)$ and $f \in H^{(0,-1)}$.
The Lagrangian $\L_{q,f}$, defined by (\ref{Lagrangian}), satisfies
\begin{align*}
D_u \L_{q,f} v &= (u - q, v)_{\omega}
+ s(u,v) + G(v,z),
\\
D_z \L_{q,f} w &= -s^*(z,w) + G(u,w) - \pair{f,w},
\end{align*}
and therefore the critical points $(u,z) \in \V_h \times \W_h$  of $\L_{q,f}$ satisfy
\begin{align}
\label{normal_eqs}
A[(u,z),(v,w)] 
= (q,v)_{\omega} + \pair{f,w},
\quad (v,w) \in \V_h \times \W_h,
\end{align}
where $A$ is the symmetric bilinear form
\begin{align}
\label{def_A}
A[(u,z),(v,w)] = &(u,v)_{\omega} + s(u,v) + G(v,z) 
- s^*(z,w) + G(u,w).
\end{align}
Note that 
$$
A[(u,z),(u,-z)] = |u|_\V^2 + \norm{u}_{\omega}^2 + \norm{z}_{\W}^2,
$$
and therefore $A$ is is weakly coercive in the following sense
\begin{align}
\label{infsup}
\tnorm{(u,z)} 
\le
C \sup_{(v,w) \in \V_h \times \W_h} 
\frac{A[(u,z),(v,w)]}{\tnorm{(v,w)}},
\quad (u,z) \in \V_h \times \W_h.
\end{align}
The Babuska-Lax-Milgram theorem implies that the equation (\ref{normal_eqs}) has a unique solution in $\V_h \times \W_h$.

\subsection{Error estimates}

In this section we show that the solution $(u_h, z_h)$ of (\ref{normal_eqs}) satisfies $u_h \to u$ in $(T_1, T_2) \times B$
as $h \to 0$, with the convergence rate (\ref{convrate_informal}).
Here $u$ is a smooth enough solution of the unstable data assimilation problem in the continuum. 
We use the shorthand notation 
$$
\norm{u}_{*} = \norm{u}_{(1,1)} + \norm{u}_{(0,2)}.
$$

\begin{lemma}
\label{tnorm_error}
Let $u \in H^{(1,1)} \cap H^{(0,2)}$
and define $f = \p_t u - \Delta u$ and $q = u|_{(0,T) \times \omega}$.
Let $(u_h, z_h) \in \V_h \times \W_h$
be the solution of (\ref{normal_eqs}).
Then there exists $C > 0$ such that for all $h \in (0,1)$,
\begin{align*}
\tnorm{(u_h - \pi_h u,z_h)} 
&\leq C h \norm{u}_{*}.
\end{align*}
\end{lemma}
\begin{proof}
The equations $\p_t u - \Delta u = f$ and $u|_\omega = q$
are equivalent with
\begin{alignat}{2}
\label{weak_cont}
&G(u,w) = \pair{f,w}, \quad &&w \in L^2(0,T; H^1_0(\Omega)),
\\\nonumber
&(q-u,v)_{\omega} = 0,
\quad &&v \in L^2((0,T) \times \omega),
\end{alignat}
and the equations (\ref{normal_eqs}) and (\ref{weak_cont})
imply for all $v \in \V_h$ and $w \in \W_h$ that 
\begin{align}
\label{A_ortho}
&A[(u_h-\pi_h u, z_h),(v,w)]
= (u-\pi_h u,v)_{\omega} 
+ G(u-\pi_h u, w) - s(\pi_h u,v).
\end{align}
By (\ref{infsup}) it is enough to show that
$$
A[(u_h - \pi_h u, z_h),(v,w)]
\le C h \norm{u}_{*} \tnorm{(v,w)}.
$$
We use (\ref{pi_h}) to bound the first term in (\ref{A_ortho}),
$$
(u - \pi_h u,v)_{\omega}
\le \norm{u - \pi_h u}_{\omega}
\norm{v}_{\omega}
\le C h \|u\|_{(0,1)} \norm{v}_{\omega},
$$
for the second term we use (\ref{ds_lower}) and (\ref{pi_h}),
$$
G(u - \pi_h u, w) 
\le C \left(\|u - \pi_h u\|_{(1,0)} + \|u - \pi_h u\|_{(0,1)}\right)\norm{w}_{\W}
\le C h \norm{u}_{*} \norm{w}_{\W},
$$
and for the third term we use (\ref{ps_upper}),
$$
s(\pi_h u, v) \le 
|\pi_h u|_\V |v|_\V \le 
C h \norm{u}_{*} |v|_\V.
$$
\end{proof}

\begin{theorem}
\label{th_main_unstable}
Let $\omega, B \subset \Omega$, $0 < T_1 < T_2 < T$
and $\kappa \in (0,1)$ be as in Theorem \ref{th_cont_unstable}.
Let $u \in H^{(1,1)} \cap H^{(0,2)}$
and define $f = \p_t u - \Delta u$ and $q = u|_{(0,T) \times \omega}$.
Let $(u_h, z_h) \in \V_h \times \W_h$
be the solution of (\ref{normal_eqs}).
Then there exists $C > 0$ such that
for all $h \in (0,1)$
\begin{align*}
\norm{u_h - u}_{L^2(T_1, T_2; H^1(B))} 
&\le C h^\kappa (\|u\|_{*} + \|f\|).
\end{align*}
\end{theorem}
\begin{proof}
%We use the same notation for norms as in the proof of  Lemma \ref{tnorm_error}.
We define the functional 
$\pair{r,w} = G(u_h - u, w)$,
$w \in H_0^{(0,1)}$.
The equation $\p_t u - \Delta u = f$ implies that 
$$
\pair{r,w} = G(u_h, w) - \pair{f, w}.
$$
By taking $v=0$ in (\ref{normal_eqs}) we get
$G(u_h, w) = \pair{f,w} + s^*(z_h,w)$,
$w \in \W_h$,
and therefore
\begin{align}
\label{r_ortho}
\pair{r,w} &= G(u_h,w) - \pair{f,w} - G(u_h, \pi_h w) + G(u_h, \pi_h w)
\\\notag&= G(u_h,w - \pi_h w) - \pair{f,w-\pi_h w} + s^*(z_h,\pi_h w),
\quad w \in H_0^{(0,1)}.
\end{align}

We arrive to the estimate
\begin{align}
\label{the_error}
\norm{r}_{(0,-1)}  
\le C (|u_h|_{\V} + \norm{z_h}_{\W} + h \|f\|)
\end{align}
after we bound the first term in (\ref{r_ortho})
by using (\ref{ps_lower})
and (\ref{pi_h_stab})-(\ref{pi_h})
\begin{align}
\label{semidisc_err_a}
G(u_h, w - \pi_h w) 
&\le
C |u_h|_\V \left(h^{-1} \|w - \pi_h w\| + \|w - \pi_h w\|_{(0,1)}\right)
\\\notag&\le C |u_h|_\V \|w\|_{(0,1)},
\end{align} 
the second term by using (\ref{pi_h})
\begin{align*}
(f, w - \pi_h w)
&\le \norm{f} \norm{w - \pi_h w}
\le C h \norm{f}\norm{w}_{(0,1)},
\end{align*}
and the last term by using (\ref{ds_upper})
\begin{align*}
s^*(z_h, \pi_h w) 
\le
\norm{z_h}_\W \norm{\pi_h w}_\W \le 
C \norm{z_h}_\W \norm{w}_{(0,1)}.
\end{align*}

We bound the first term in (\ref{the_error}) by using Lemma \ref{tnorm_error} and (\ref{ps_upper}) 
\begin{align*}
|u_h|_\V \le |u_h - \pi_h u|_\V + |\pi_h u|_\V
\le C h \norm{u}_{*},
\end{align*}
and observe that the analogous bound for the second term
follows immediately from Lemma \ref{tnorm_error}.
Thus 
\begin{align*}
\norm{r}_{(0,-1)}  
\le C h(\norm{u}_{*} + \|f\|).
\end{align*}
By applying Theorem \ref{th_cont_unstable}
to $u_h - u$, and using Lemma \ref{tnorm_error} 
and (\ref{pi_h}) to bound the term 
$$
\norm{u_h - u}_\omega \le \norm{u_h - \pi_h u}_\omega + \norm{\pi_h u - u}_\omega \le C h \norm{u}_{*} + C h \norm{u}_{(0,1)}
$$
we see that
\begin{align*}
&\norm{u_h - u}_{L^2(T_1,T_2; H^1(B))} 
\le C h^\kappa (\|u\|_{*} + \|f\|)^\kappa (\norm{u_h - u} + 
h(\|u\|_{*} + \|f\|))^{1-\kappa}.
\end{align*}
It remains to show the bound 
\begin{align}
\label{apriori_bound}
\norm{u_h-u} \le C \norm{u}_{*}.
\end{align}
The Poincar\'e type inequality (\ref{ps_poincare}) implies 
\begin{align*}
\norm{u_h - \pi_h u} \le C h^{-1} |u_h - \pi_h u|_{\V} + C h^{-1} \norm{u_h - \pi_h u}_\omega
\le C \norm{u}_{*},
\end{align*}
where we used again Lemma \ref{tnorm_error}.
The inequality (\ref{apriori_bound}) follows since $\norm{\pi_h u - u}$
can be bounded by $\norm{u}_{*}$, in fact, even by $C h^2 \norm{u}_{*}$.
\end{proof}

\subsection{The effect of perturbations in data}
In practice the data $f,q$ in \eqref{the_data} are typically polluted by measurement
errors. Such effects can easily be included in the analysis and we
show in this sections how the above results must be modified to
account for this case. We consider the case where $f,q$ in
\eqref{normal_eqs} are replaced by the perturbed counterparts
%\HOX{$\delta q$ would be more natural in my opinion}
\begin{equation}\label{perturbed_data}
\tilde q = u|_{(0, T) \times \omega} + \delta q, \quad \tilde f = f + \delta f, 
\end{equation}
where $u$ is the solution to the underlying unpolluted problem, 
that is, $\p_t u - \Delta u = f$.
We assume that the perturbations $\delta q, \delta f$
satisfy $\delta f \in H^{-1}(\Omega)$ and $\delta q \in
(H^1(\omega))'$, and use the following
measure of the size of the perturbation
\[
\delta(\tilde q,\tilde f) = \|\delta q\|_{L^2(0,T;(H^1(\omega))')} + \|\delta f\|_{(0,-1)}.
\] 
%where $\|\delta u\|_{(H^1(\omega))'} = \sup_{v \in H^1(\Omega)}
%\pair{\delta u,
%  w}_{[(H^1(\omega))',H^1(\omega)]}/\|v\|_{H^1(\omega)}$. 

%Let us consider the model case in Section \ref{sec_model_case}.
The critical
points $(u,z) \in \mathcal{V}_h \times \mathcal{W}_h$ of the perturbed
Lagrangian $\L_{\tilde q, \tilde f}$ satisfy
\begin{align}
\label{normal_eqs_pert}
A[(u,z),(v,w)] 
= \pair{\tilde q,v}_\omega + \langle \tilde f,w \rangle,
\quad (v,w) \in \V_h \times \W_h,
\end{align}
where $\pair{\cdot, \cdot}_\omega$ is the dual pairing between 
$L^2(0,T;(H^1(\omega))')$ and $L^2(0,T;H^1(\omega))$.
We have chosen the perturbations in the weakest spaces for which the
formulation \eqref{normal_eqs_pert} remains well defined. It should be noted however that
smoother perturbations, for instance in $H^{(0,0)}$, do not lead to
improved estimates.
First we consider again the
residual quantities as in Lemma \ref{tnorm_error}.

\begin{lemma}
\label{tnorm_error_pert}
Let $u \in H^{(1,1)} \cap H^{(0,2)}$
and define $\tilde f $ and $\tilde q$ by \eqref{perturbed_data}.
Let $(u_h, z_h) \in \V_h \times \W_h$
be the solution of (\ref{normal_eqs_pert}).
Then there exists $C > 0$ such that for all $h \in (0,1)$
\begin{align*}
\tnorm{(u_h - \pi_h u,z_h)} 
&\leq C( h \norm{u}_{*} + \delta(\tilde q,\tilde f)).
\end{align*}
\end{lemma}
\begin{proof}
Proceeding as in the proof of Lemma \ref{tnorm_error},
 the equations (\ref{normal_eqs_pert}) and (\ref{weak_cont}) 
imply for all $v \in \V_h$ and $w \in \W_h$ that 
\begin{align}
\label{A_ortho_pert}
&A[(u_h-\pi_h u, z_h),(v,w)]
= (u-\pi_h u,v)_{\omega} +\pair{\delta q, w}_\omega\\\notag
& \quad + G(u-\pi_h u, w)+\pair{\delta f,w}- s(\pi_h u,v).
\end{align}
The only terms that are not covered by the 
previous analysis are those with the perturbations, and we have
\[
\pair{\delta q, w}_\omega +  \pair{\delta f,w}
\leq \delta(\tilde q,\tilde f) \|w\|_{(0,1)}.
%\leq  \delta(\tilde q,\tilde f) \tnorm{v}_h
\]
The choice (\ref{stab_simple}) of the dual stabilizer $s^*$
guarantees that $\|w\|_{(0,1)} \le \norm{w}_\W$.
\end{proof}
The error estimate of Theorem \ref{th_main_unstable} may now be modified to include
the perturbations. % by using the modified consistency relation and Lemma \ref{tnorm_error_pert}.
\begin{theorem}
\label{th_main_unstable_pert}
Let $\omega, B \subset \Omega$, $0 < T_1 < T_2 < T$
and $\kappa \in (0,1)$ be as in Theorem \ref{th_cont_unstable}.
Let $u \in H^{(1,1)} \cap H^{(0,2)}$
and define $\tilde f $ and $\tilde q$ by \eqref{perturbed_data}.
Let $(u_h, z_h) \in \V_h \times \W_h$
be the solution of (\ref{normal_eqs_pert}).
Then there exists $C > 0$ such that
for all $h \in (0,1)$
\begin{align*}
&\norm{u_h - u}_{L^2(T_1, T_2; H^1(B))} 
\\&\quad\le C (h(\norm{u}_{*} + \|f\|)+ \delta(\tilde q,\tilde f))^\kappa
(\norm{u}_{*} + h \|f\| + h^{-1} \delta(\tilde q,\tilde f))^{(1-\kappa)}.
\end{align*}
\end{theorem}
\begin{proof}
%We use the same notation for norms as in the proof of  Lemma \ref{tnorm_error}.
We proceed as in the proof of Theorem \ref{th_main_unstable}, but due to
the perturbation, the Galerkin orthogonality (\ref{r_ortho}) no longer holds exactly and we obtain
\begin{align*}
%\label{r_ortho_pert}
\pair{r,w} 
%&= G(u_h,w) - \pair{f,w} - G(u_h, \pi_h w) + G(u_h, \pi_h w)
%\\\notag
&= G(u_h,w - \pi_h w) - \pair{f,w-\pi_h w} + \pair{\delta
          f,\pi_h w} + s^*(z_h,\pi_h w),
\end{align*}
for $w \in H_0^{(0,1)}$.
Analogously to the proof of Theorem \ref{th_main_unstable}, we obtain
\begin{align}
\label{the_error_pert}
\norm{r}_{(0,-1)}  
\le C (|u_h|_{\V} + \norm{z_h}_{\W} + h \|f\| + \|\delta f\|_{(0,-1)}).
\end{align}
We bound the first term in (\ref{the_error_pert}) by using Lemma \ref{tnorm_error_pert} and (\ref{ps_upper}) 
\begin{align*}
|u_h|_\V \le |u_h - \pi_h u|_\V + |\pi_h u|_\V
\le C( h \norm{u}_{*} + \delta(\tilde q,\tilde f)),
\end{align*}
and the second term by using again Lemma \ref{tnorm_error_pert}.
Thus 
\begin{align*}
\norm{r}_{(0,-1)}  
\le C h(\norm{u}_{*} + \|f\|)+ C\delta(\tilde q,\tilde f) .
\end{align*}
As before, applying Theorem \ref{th_cont_unstable}
to $u_h - u$, and using Lemma \ref{tnorm_error_pert} 
and (\ref{pi_h}) to bound the term 
$\norm{u_h - u}_\omega$
leads to
\begin{align*}
&\norm{u_h - u}_{L^2(T_1,T_2; H^1(B)} 
\\&\quad\le C (h (\|u\|_{*} + \|f\|) + \delta(\tilde q,\tilde f))^\kappa (\norm{u_h - u} + 
h(\|u\|_{*} + \|f\|) + \delta(\tilde q,\tilde f))^{1-\kappa}.
\end{align*}
It remains to bound $\norm{u_h - u}$.
The Poincar\'e type inequality  (\ref{ps_poincare}) implies 
\begin{align*}
\norm{u_h - \pi_h u} &\le C h^{-1} |u_h - \pi_h u|_{\V} + C h^{-1} \norm{u_h - \pi_h u}_\omega\\
&\le C (\norm{u}_{*} + h^{-1} \delta(\tilde q,\tilde f)),
\end{align*}
where we used again Lemma \ref{tnorm_error_pert}.
\end{proof}

According to Theorem \ref{th_main_unstable_pert}, the error 
of the numerical approximation
can diverge if no function $\tilde u$ exists
such that $\p_t \tilde u - \Delta \tilde u = \tilde f$ and $\tilde u\vert_{(0,T) \times \omega} = \tilde
q$. On the other hand if the perturbation is known to satisfy an upper
bound, making it smaller than the discretization error on certain scales, we can recover
an error estimate in the pre-asymptotic regime.
\begin{corollary}
Under the assumptions of Theorem \ref{th_main_unstable_pert} and
assuming that there exists $h_0>0$ such that
\begin{equation}\label{pert_bound}
\delta(\tilde q,\tilde f) \leq h_0 (\|u\|_{*} + \|f\|),
\end{equation}
then for $h>h_0$ there holds
\begin{align*}
\norm{u_h - u}_{L^2(T_1, T_2; H^1(B))} 
&\le C h^\kappa (\|u\|_{*} + \|f\|)^\kappa.
\end{align*}
\end{corollary}
%\begin{proof}
%Immediate by applying the bound \eqref{pert_bound} in the estimate of
%Theorem \ref{th_main_unstable_pert}.
%\end{proof}

\section{A framework for stabilization}
\label{sec_framework}

Before proceeding to the stable problem, we introduce an abstract stabilization framework based on the essential features of the model case in Section \ref{sec_model_case}.

Let $s$ and $s^*$ be bilinear forms on the spaces $\V_h$ and $\W_h$, respectively. Let $|\cdot|_\V$ be a semi-norm on $\V_h$ and let $\norm{\cdot}_\W$ be a norm on $\W_h$.
We relax (\ref{stab_norm_simple}) by requiring only that $s$ and $s^*$ are continuous with respect to $|\cdot|_\V$ and $\norm{\cdot}_\W$, that is, 
\begin{align}
\label{s_cont}
s(u,u) \le C |u|_\V^2, \quad s^*(z,z) \le C \norm{z}_\W^2, \quad u \in \V_h,\ z \in \W_h,\ h > 0. 
\end{align}
Let $\norm{\cdot}_*$ be the norm of a continuously embedded subspace $H^*$ of the energy space (\ref{energy_space_nobc}). The space $H^*$ encodes the apriori smoothness. 
We relax the lower bounds (\ref{ps_lower}) and (\ref{ds_lower}) as follows
\begin{align}
\label{V_lower}
G(u,z - \pi_h z) 
&\le C |u|_\V \|z\|_{(0,1)},
& u \in \V_h,\ z \in H_0^{(0,1)},
\\\label{W_lower}
G(u - \pi_h u,z)  
&\le C h \norm{z}_{\W} \|u\|_{*},
& u \in H^{*},\ z \in \W_h,
\end{align}
where $\pi_h$ is an interpolator satisfying (\ref{pi_h_stab})-(\ref{pi_bc}).
We replace the upper bound (\ref{ps_upper})
by its abstract analogue
\begin{align}
\label{Vs_upper}
|\pi_h u|_\V &\le C h \norm{u}_{*}, &u \in H^{*},
\end{align}
and require that (\ref{ds_upper}) and (\ref{ps_poincare}) hold as before.
Finally, in the abstract setting, we assume that the weak coercivity (\ref{infsup}) holds where $A$ and $\tnorm{\cdot}$ are defined as above. 

It can be verified that the following analogue of Theorem \ref{th_main_unstable} holds under these assumptions.

\begin{theorem}
\label{th_main_unstable_abs}
Let $\omega, B \subset \Omega$, $0 < T_1 < T_2 < T$
and $\kappa \in (0,1)$ be as in Theorem \ref{th_cont_unstable}.
Let $u \in H^{*}$
and define $f = \p_t u - \Delta u$ and $q = u|_\omega$.
Suppose that the primal and dual stabilizers
satisfy 
(\ref{s_cont})-(\ref{Vs_upper}), (\ref{ds_upper}), (\ref{ps_poincare}) and (\ref{infsup}).
Then (\ref{normal_eqs}) has a unique solution  $(u_h, z_h) \in \V_h \times \W_h$, and there exists $C > 0$ such that
for all $h \in (0,1)$
\begin{align*}
\norm{u_h - u}_{L^2(T_1, T_2; H^1(B))} 
&\le C h^\kappa (\|u\|_{*} + \|f\|).
\end{align*}
\end{theorem}

\section{The stable problem}

Let us now consider the stable problem with the additional lateral boundary condition $u|_{(0,T) \times \partial\Omega} = 0$. 
We may impose the boundary condition on the discrete level by changing the space $\V_h$, defined previously by (\ref{stab_spaces_simple}), to 
$$
\V_h = H^1(0,T; W_h).
$$
In the stable case we do not need an inequality analogous to (\ref{apriori_bound}) since the estimate in Theorem \ref{th_cont_stable}
does not contain an apriori term, that is, a term analogous to 
$\norm{u}_{L^2((0,T) \times \Omega)}$ in Theorem \ref{th_cont_unstable}.
In the previous section, the Poincar\'e type inequality (\ref{ps_poincare}) was used only to obtain (\ref{apriori_bound}),
and whence we can relax the conditions there by dropping (\ref{ps_poincare}). However, we still need to require that 
$\tnorm{\cdot}$ is a norm on $\V_h \times \W_h$.

We will now proceed to a concrete case. 
The choices in Section \ref{sec_framework} work also in the stable case, however, the additional structure allows us to choose a weaker stabilization. The use of a weaker stabilization is motivated by the fact that it leads to a weaker coupling of the two heat equations in (\ref{normal_eqs}), which again can be exploited when solving (\ref{normal_eqs}) in practice. 

The stabilizers and semi-norms are chosen as follows,
\begin{align}
\label{stab_2}
&s(u,u) = \norm{h \nabla u(0, \cdot)}_{L^2(\Omega)}^2,
\quad s^* = a,
\\\label{stab_norm_2}
&|u|_\V = s(u,u)^{1/2} + \norm{h \p_t u}, \quad \norm{z}_\W = s^*(z,z)^{1/2},
\end{align}
and we define $H^* = H_0^{(1,1)}$.
To counter the lack of primal stabilization on most of the cylinder $(0,T) \times \Omega$, we choose $\pi_h$ to be the orthogonal projection 
$
\pi_h : H_0^1(\Omega) \to W_h
$
with respect to the inner product $(\nabla u, \nabla v)_{L^2(\Omega)}$. That is, $\pi_h$ is defined by
\begin{align}
\label{def_pi2}
(\nabla \pi_h u, \nabla v)_{L^2(\Omega)} = (\nabla u, \nabla v)_{L^2(\Omega)}, \quad u \in H_0^1(\Omega),\ v \in W_h.
\end{align}
As $\Omega$ is a convex polyhedron,
this choice satisfies the estimates (\ref{pi_h_stab})-(\ref{pi_h}),
see e.g. \cite[Th. 3.12-18]{Ern2004}.

\begin{lemma}
The choices (\ref{stab_2})-(\ref{def_pi2}) satisfy the properties
(\ref{s_cont})-(\ref{Vs_upper}), (\ref{ds_upper}) and (\ref{infsup}).
Moreover, $\tnorm{\cdot}$ is a norm on $\V_h \times \W_h$.
\end{lemma}
\begin{proof}
It is clear that the continuities (\ref{s_cont}) hold.
We begin with the lower bound (\ref{V_lower}). By the orthogonality (\ref{def_pi2}),
$$
G(u, z- \pi_h z) = (\p_t u, z- \pi_h z) 
\le \norm{h \p_t u} h^{-1} \norm{z- \pi_h z} \le C \norm{h \p_t u} \norm{z}_{(0,1)}.
$$

Towards the lower bound (\ref{W_lower}),
we use the orthogonality (\ref{def_pi2}) as above, 
$$
G(u - \pi_h u, z) = (\p_t u - \pi_h \p_t u, z)
\le C h \norm{u}_{(1,1)} \norm{z}. 
$$
The bound (\ref{W_lower}) follows from the Poincar\'e inequality (\ref{ds_poincare}).

The bound (\ref{Vs_upper}) follows from the continuity of the trace
\begin{align}
\label{trace_time}
\norm{\nabla u(0, \cdot)}_{L^2(\Omega)} \le \norm{u}_{(1,1)},
\end{align}
and the continuity of the projection $\pi_h$.
The bound (\ref{ds_upper}) follows immediately from the continuity of $\pi_h$.

We turn to the weak coercivity (\ref{infsup}).
The essential difference between the unstable and the stable case is that in the latter case $\p_t u \in \W_h$ when $u \in \V_h$.
We have 
$$
A[(u,z),(0,\p_t u)] = - s^*(z,\p_t u) + G(u,\p_t u)
=  \norm{\p_t u}^2 + a(u, \p_t u) - a(z, \p_t u),
$$
and thus using bilinearity of $A$,
\begin{align}
\label{A_expansion}
A[(u,z),(u,\alpha h^2 \p_t u - z)]
&= s(u,u) + \alpha \norm{h \p_t u}^2 + \norm{u}_\omega^2 + s^*(z,z)
\\\notag&\quad + \alpha h^2 a(u, \p_t u) - \alpha h^2 a(z, \p_t u),
\end{align}
where $\alpha > 0$.
We will establish (\ref{infsup}) 
by showing that there is $\alpha \in (0,1)$ such that 
\begin{align}
\label{infsup1_step1}
\tnorm{(u, w - z)} &\le C\tnorm{(u,z)},
\\
\label{infsup1_step2}
\tnorm{(u,z)}^2 &\le C A[(u,z),(u,w-z)],
\end{align}
where $w = \alpha h^2 \p_t u$.

Towards (\ref{infsup1_step1}) we observe that 
\begin{align*}
\tnorm{(u,w-z)}^2
&= 
\tnorm{(u,z)}^2 - 2 s^*(z,w) + s^*(w,w)
\le 2 \tnorm{(u,z)}^2 + 2  s^*(w,w).
\end{align*}
We use the discrete inverse inequality (\ref{inverse_disc}) to bound the second term
$$
s^*(w,w) = \alpha^2 h^4 \norm{\p_t \nabla u}^2 \le C \alpha^2 h^2 \norm{\p_t u}^2
\le C \alpha^2 \tnorm{(u,z)}^2, \quad \alpha > 0.
$$

It remains to show (\ref{infsup1_step2}). 
Towards bounding the first cross term in (\ref{A_expansion}) we observe that  
\begin{align*}
2 a(u, \p_t u) 
&= 
\int_0^T \p_t \norm{\nabla u(t, \cdot)}_{L^2(\Omega)}^2 dt 
= 
\norm{\nabla u(T, \cdot)}_{L^2(\Omega)}^2 - \norm{\nabla u(0, \cdot)}_{L^2(\Omega)}^2.
\end{align*}
Hence
$
\alpha h^2 a(u, \p_t u) \ge -\alpha s(u,u) / 2.
$
We use the arithmetic-geometric inequality,  
$$
ab \le (4\epsilon)^{-1} a^2 + \epsilon b^2, \quad a,b \in \R,\ \epsilon > 0,
$$
and the discrete inverse inequality (\ref{inverse_disc})
to bound the second cross term in (\ref{A_expansion}),
$$
\alpha h^2 a(z, \p_t u) 
\le \alpha (4\epsilon)^{-1} a(z,z) 
+ \alpha \epsilon h^4 \norm{\p_t \nabla u}^2 \le
\alpha (4\epsilon)^{-1} a(z,z)  + C \alpha \epsilon \norm{h \p_t u}^2.
$$
Choosing $\epsilon = 1 / (2 C)$ we obtain %, with a new constant $C > 0$,
$$
A[(u,z),(u,w - z)] \ge 
(1-\alpha/2) s(u,u) + \alpha \norm{h \p_t u}^2 / 2 + \norm{u}_\omega^2 + (1-C \alpha/2) s^*(z,z),
$$
and therefore (\ref{infsup1_step2}) holds with small enough $\alpha > 0$. 

Finally, using the Poincar\'e inequality (\ref{ds_poincare}),  we see that $\tnorm{(u,z)} = 0$ implies $z=0$ and $u(0, \cdot) = 0$.
As also $\p_t u = 0$, we have $u=0$, and therefore $\tnorm{\cdot}$ is a norm. 
\end{proof}

In the stable case we have the following convergence rate. 

\begin{theorem}
\label{th_main_stable_abs}
Let $\omega \subset \Omega$ be open and non-empty and let $0 < T_1 < T$.
Suppose that (A2) holds. Let $u \in H^{*}$
and define $f = \p_t u - \Delta u$ and $q = u|_\omega$.
Suppose that the primal and dual stabilizers
satisfy 
(\ref{s_cont})-(\ref{Vs_upper}), (\ref{ds_upper}) and (\ref{infsup}).
Then (\ref{normal_eqs}) has a unique solution  $(u_h, z_h) \in \V_h \times \W_h$, and there exists $C > 0$ such that
for all $h \in (0,1)$
\begin{align*}
&\norm{u_h - u}_{C(T_1, T; L^2(\Omega))} +
\norm{u_h - u}_{L^2(T_1, T; H^1(\Omega))} +
\norm{u_h - u}_{H^1(T_1, T; H^{-1}(\Omega))} 
\\&\quad\le 
C h (\|u\|_{*} + \|f\|).
\end{align*}
\end{theorem}

The proof is analogous to that of Theorem \ref{th_main_unstable}, however, we give it here for the sake of completeness. 

\begin{proof}
We begin again by showing the estimate
\begin{align}
\label{tnorm_err_stable}
\tnorm{(u_h - \pi_h u,z_h)} 
&\leq C h \norm{u}_{*}.
\end{align}
The weak coercivity (\ref{infsup}) implies that 
in order to show (\ref{tnorm_err_stable}) it is enough bound the three terms on the right hand side of (\ref{A_ortho}). % by $C h \norm{u}_{*} \tnorm{(v,w)}$.
For the first of them, that is, $(u - \pi_h u,v)_{\omega}$,
we use (\ref{pi_h}) as in the proof of Lemma \ref{tnorm_error}.
The lower bound (\ref{W_lower}) applies to the second term
$G(u - \pi_h u, w)$, 
and for the third one we use the continuity (\ref{s_cont})
and the upper bound (\ref{Vs_upper}),
$$
s(\pi_h u, v) \le 
C |\pi_h u|_\V |v|_\V \le 
C h \norm{u}_{*} |v|_\V.
$$

We define the residual $r$ as in the proof of Theorem \ref{th_main_unstable}, and show next that (\ref{the_error}) holds.
It is enough to bound the three terms on the right hand side of (\ref{r_ortho}).
The lower bound (\ref{V_lower}) applies to the first term
$G(u_h, w - \pi_h w)$, for the second term $(f, w - \pi_h w)$
we use (\ref{pi_h}) as in the proof of Theorem \ref{th_main_unstable},
and for the third term we use the continuity (\ref{s_cont})
and the upper bound (\ref{ds_upper})
\begin{align*}
s^*(z_h, \pi_h w) 
\le
C \norm{z_h}_\W \norm{\pi_h w}_\W \le 
C \norm{z_h}_\W \norm{w}_{(0,1)}.
\end{align*}

The inequalities (\ref{the_error}), (\ref{tnorm_err_stable}) and (\ref{Vs_upper})
imply
$$
\norm{r}_{(0,-1)} 
\le 
C (|u_h - \pi_h u|_{\V} + |\pi_h u|_\V + \norm{z_h}_{\W} + h \|f\|)
\le 
C h (\norm{u}_* + \norm{f}).
$$
Finally, Theorem \ref{th_cont_stable} implies that 
\begin{align*}
&\norm{u_h - u}_{C(T_1, T; L^2(\Omega))} +
\norm{u_h - u}_{L^2(T_1, T; H^1(\Omega))} +
\norm{u_h - u}_{H^1(T_1, T; H^{-1}(\Omega))} 
\\&\quad\le C \norm{u_h - u}_\omega + 
C h (\|u\|_{*} + \|f\|).
\end{align*}
The claim follows by using (\ref{tnorm_err_stable}) and (\ref{pi_h}),
$$
\norm{u_h - u}_\omega
\le \norm{u_h - \pi_h u}_\omega + \norm{\pi_h u - u}_\omega
\le C h \norm{u}_*.
$$
Here we used also the assumption that $H^*$ is a continuously embedded subspace of the energy space (\ref{energy_space_nobc}), namely, this implies that the embedding $H^* \subset H^{(0,1)}$ is continuous. 
\end{proof}
\begin{remark}
If the data $q,f$ is perturbed in the stable case, the data assimilation
problem behaves like a typical well posed problem, that is, the term
$\delta(\tilde q, \tilde f)$ needs to be added on the right-hand side of the estimate
in Theorem 6, but this time without any negative power of $h$.
\end{remark}
\section{Conclusion}
In the present paper our aim was to show how methods known from the
theory of stabilized finite element methods can be applied to the
design of computational methods for non-stationary data assimilation problems by
first discretizing and then regularizing the corresponding 4DVAR
optimization system. 

A key feature of this framework is that the error
analysis is based on numerical stability of residual quantities that
are independent of the stability properties of the continuous
model. The residual quantities are then bounded by using
conditional stability estimates, based on Carleman estimates, to derive
error estimates that reflect the approximation properties of the
finite element space and the stability of the continuous problem in an
optimal way. An upshot of this approach is that
it gives a clear lead on how to design regularization
for a given problem in order to obtain the best accuracy of
the approximation with the
least effect of perturbation in data. 

Observe that the stabilization
operators proposed herein are not unique, for instance it is
straightforward to show that the dual stabilizer in equation
\eqref{stab_simple} may be chosen as the first part of the primal
stabilizer, leading to similar error estimates for unperturbed data. The error
estimates also gives an indication on what form Carleman estimates should take to make them immediately applicable for error analysis. Fully discrete schemes can be treated in a similar way and will be considered in a forthcoming paper.

\section*{Appendix. Continuum estimates}

\def\S{\mathcal S_\Delta}
\def\A{\mathcal A_\Delta}
\def\P{\mathcal P_\Delta}
\def\RR{\mathcal R_\Delta}
\def\B{\mathcal B_\Delta}
\def\s{s}
\def\Q{q}

\def\L{L}
\def\Sh{\mathcal S_L}
\def\Ah{\mathcal A_L}
\def\RRh{\mathcal R_L}
\def\Bh{\mathcal B_L}
\def\bh{\beta_L}
\def\Qh{q}

Theorem \ref{th_cont_unstable} is a consequence of the 
so-called three cylinders inequality that goes back to \cite{Glagoleva1965},
see \cite{Vessella2008,Yamamoto2009} for an overview of the related literature.
The variant of the inequality, needed in the above convergence analysis,
seems not to appear in the literature, and we prove it here by using the Carleman estimate \cite{Isakov1993}.

Theorem \ref{th_cont_stable} is a variant of \cite{Emanuilov1995}, the difference being that in \cite{Emanuilov1995} the domain $\Omega$ is assumed to have $C^2$-smooth boundary. 
We outline below the modifications needed in the case that $\Omega$ is a convex polyhedron.

We use the notation $B(x,r) = \{y \in \Omega;\ d(y,x) < r\}$, $x \in \Omega$, $r>0$, where $d$ is the Euclidean distance.

\begin{theorem}[Three cylinders inequality]
\label{th_3_cylinders}
Let $x_0 \in \Omega$ and $0 < r_1 < r_2 < d(x_0, \p \Omega)$.
Write $B_j = B(x_0, r_j)$, $j=1,2$.
Let $T > 0$ and $0 < \epsilon < T$.
Then there are $C > 0$ and $\kappa \in (0,1)$
such that for all $u \in C^2(\R \times \Omega)$
$$
\norm{u}_{L^2(\epsilon, T - \epsilon; H^1(B_2))} 
\le 
C (\norm{u}_{L^2(0, T; H^1(B_1))} + 
\norm{L u}_{L^2((0, T) \times \Omega)})^\kappa \norm{u}_{L^2(0, T;H^1(\Omega))}^{1-\kappa}.
$$
\end{theorem}
\begin{proof}
Let $0 < r_0 < r_1$ and $r_2 < r_3 < r_4 < d(x_0, \p \Omega)$.
Define $B_j = B(x_0, r_j)$, $j=0,3,4$.
We choose non-positive $\rho_1 \in C^\infty(\Omega)$ 
such that $\rho > - r_0$ in $B_0$ and that $\rho_1(x) = -d(x,x_0)$ outside $B_0$.
Define $I_1 = (\epsilon,T-\epsilon)$ and $I_2 = (\epsilon/2,T-\epsilon/2)$, and choose non-positive $\rho_2 \in C^\infty(\R)$
such that $\rho_2 \le -r_3$ outside $I_2$ and
$\rho_2 = 0$ in $I_1$.
We define $\rho(t,x) = \rho_1(x) + \rho_2(t)$
and $\phi = e^{\alpha \rho}$ where $\alpha > 0$.

We will apply \cite[Th. 1.1]{Isakov1993}
to the heat operator $L = \p_t - \Delta$,
 with the above weight function $\phi$. Condition (1.5) in the theorem reduces to 
\begin{align}
\label{pseudoconvexity}
D^2 \phi(\zeta, \overline \zeta) > 0,
\quad \zeta = \xi + i \tau \nabla \phi,\ |\xi| = \tau |\nabla \phi|,\ \tau > 0,\ \xi \in \R^n,
\end{align}
where $D^2 \phi$ is the Hessian of $\phi$ with respect to $x$.
We have $\nabla \phi = \alpha \phi \nabla \rho$ and
\begin{align*}
D^2 \phi(\zeta, \overline \zeta) 
%&= \alpha \phi (\alpha \nabla \rho \otimes \nabla \rho + D^2 \rho)(\zeta, \overline \zeta)
&= \alpha \phi (\alpha |\nabla \rho \cdot \zeta|^2 + D^2 \rho(\zeta, \overline \zeta))
\\&\ge \alpha \phi ( \alpha (\tau \alpha \phi |\nabla \rho|^2)^2 
- \mu (\tau \alpha \phi |\nabla \rho|)^2),
\end{align*}
where $\mu$ is a constant depending only on $\rho$.
As $|\nabla \rho| = 1$ outside $B_0$, condition (\ref{pseudoconvexity})
holds in $\overline{B_4} \setminus B_0$ for large enough $\alpha > 0$.
Let $\chi \in C_0^\infty((0,T) \times (B_4 \setminus B_0))$ satisfy 
$\chi = 1$ in $I_2 \times (B_3 \setminus B_1)$, 
and set $w = \chi u$. Then \cite[Th. 1.1]{Isakov1993} implies that for large $\tau > 0$,
\begin{align}
\label{cylinders_step1}
&\int_0^T \int_{B_4 \setminus B_0} ( \tau |\nabla w|^2 
+ \tau^3 |w|^2) e^{2 \tau \phi}\, dx dt
%\\\notag&\quad
\le C \int_0^T \int_{B_4 \setminus B_0} |L w|^2 e^{2 \tau \phi}\, dx dt.
\end{align}
%Observe that $|P w|^2 \le C |P u|^2 + |[P, \chi] u|^2$.

We define $\Phi(r) = e^{-\alpha r}$, $I = (0,T)$ 
and 
$$
Q_1 = I_2 \times B_1,
\quad
Q_2 = ((I \setminus I_2) \times B_4) \cup (I \times (B_4 \setminus B_3)).
$$
Observe that $\phi \le \Phi(r_3)$
in $Q_2$, and that the commutator $[L, \chi]$ vanishes outside $Q_1 \cup Q_2$.
Hence the right-hand side of (\ref{cylinders_step1})
is bounded by a constant times
\begin{align}
\label{cylinders_step2}
&\int_{(0,T) \times B_4} |L u|^2 e^{2 \tau \phi} dx dt
+ \int_{Q_1 \cup Q_2} |[L, \chi]u|^2 e^{2 \tau \phi} dx dt
\\\notag&\quad\le 
e^{2 \tau} \norm{Lu}_{L^2((0,T) \times B_4)}^2 
+ e^{2 \tau} \norm{u}_{L^2(0,T; H^1(B_1))}^2 
+ e^{2\tau \Phi(r_3)} 
\norm{u}_{L^2(0,T; H^1(B_4))}^2.
\end{align}

The left-hand side of (\ref{cylinders_step1})
is bounded  from below by
\begin{align}
\label{cylinders_step3}
%\int_{B_3 \setminus B_1} ( \lambda |\nabla w|^2 
%+ \lambda^3 |w|^2) e^{2\ell}\, dx 
%&\ge 
&\int_{I_1 \times (B_2 \setminus B_1)} 
\left( \tau |\nabla u|^2 
+ \tau^3 |u|^2 \right) e^{2 \tau \phi}\, dx dt
\ge 
e^{2 \tau \Phi(r_2)} \norm{u}_{L^2(\epsilon, T-\epsilon; H^1(B_2 \setminus B_1))}^2.
\end{align}
The inequalities (\ref{cylinders_step1})-(\ref{cylinders_step3}) imply
$$
\norm{u}_{L^2(\epsilon, T-\epsilon; H^1(B_2))}
\le 
e^{\tau} 
\left(\norm{L u}_{L^2((0,T) \times B_4)} + 
\norm{u}_{L^2(0,T; H^1(B_1))} \right)
+ e^{-p \tau} \norm{u}_{L^2(0,T; H^1(B_4))}.
$$
where $p = \Phi(r_2) - \Phi(r_3) > 0$.
The claim follows from \cite[Lemma 5.2]{LeRousseau2012}.
\end{proof}

\begin{proof}[Proof of Theorem \ref{th_cont_unstable}]
We use notation from the proof of Theorem \ref{th_3_cylinders}.
By replacing $\omega$ with a smaller set we may assume without loss of generality that it is a ball of the above form $B_1$.
We will show a local version of the claimed estimate where $B$ is replaced by a ball of the form $B_2$. The general case follows by 
covering $B$ by finite chains of balls starting from $\omega$, and by iterating the local result.

Let us first consider the case $u \in C^\infty(\R \times \Omega)$.
We choose 
$0 < r_0 < r_1$ and 
$r_2 < r_3 < d(x_0, \Omega)$ 
and write $B_j = B(x_0, r_j)$, $j=0,3$.
Let $0 < \epsilon_0 < \epsilon$
and choose $\eta \in C^\infty(\R)$ such that $\eta = 0$
near the origin and $\eta(t) = 1$ for $t > \epsilon_0$.
Let $w$ be the solution of 
\begin{align*}
&L w = \eta L u\quad \text{in $(0,T) \times B_3$}, 
\\
&w|_{x \in \p B_3} = 0, \quad w|_{t=0} = 0,
\end{align*}
and set $v = u - w$. 
Note that  $w \in C^\infty([\epsilon_0, T] \times \Omega)$.
Theorem \ref{th_3_cylinders} with $u$ replaced by $v$, $B_1$ by $B_0$, $\Omega$ by $B_3$, and $t=0$ by $t=\epsilon_0$
implies that 
\begin{align*}
&\norm{v}_{L^2(\epsilon, T - \epsilon; H^1(B_2))} 
\le 
C \norm{v}_{L^2(\epsilon_0, T; H^1(B_0))}
^\kappa \norm{v}_{L^2(\epsilon_0, T;H^1(B_3))}^{1-\kappa}
\\&\quad\le
C (\norm{u}_{L^2(\epsilon_0, T; H^1(B_0))} + \norm{L u}_{0,-1})^\kappa
(\norm{u}_{L^2(\epsilon_0, T;H^1(B_3))} + \norm{L u}_{0,-1})^{1-\kappa}.
\end{align*}
Here we have applied the energy estimate for the heat equation to $w$, see (\ref{energy_est_heat}) below.
Moreover, 
\begin{align*}
\norm{u}_{L^2(\epsilon, T-\epsilon; H^1(B_2))}
&\le \norm{v}_{L^2(\epsilon, T-\epsilon; H^1(B_2))} + C\norm{L u}_{0,-1}.
\end{align*}
We choose $\chi \in C_0^\infty((0,T) \times B_1)$ such that $\chi = 1$ in $(\epsilon_0, T-\epsilon_0) \times B_0$. Then $\chi u$ satisfies 
$$
L (\chi u) = \chi L u + [\chi, L] u, \quad (\chi u)|_{x \in \p B_1} = 0, \quad (\chi u)|_{t=0} = 0,
$$ 
and as the commutator $[\chi, L]$ is of first order in space and zeroth order in time,
\begin{align*}
\norm{u}_{L^2(\epsilon_0, T-\epsilon_0; H^1(B_0))} \le \norm{\chi u}_{L^2(0,T; H^1(B_1))}
\le C \norm{L u}_{0,-1} + C \norm{u}_{L^2((0,T) \times B_1)}.
\end{align*}
Analogously,
\begin{align*}
\norm{u}_{L^2(\epsilon_0, T-\epsilon_0; H^1(B_3))} 
\le 
C \norm{L u}_{0,-1} + C \norm{u}_{L^2((0,T) \times \Omega)}.
\end{align*}
and we have shown the local estimate in the case that $u$ is smooth. 
To conclude we observe that smooth functions are dense in the space (\ref{energy_space_nobc}).
\end{proof}

Let us now turn to 
Theorem \ref{th_cont_stable}.  
In the case of convex, polygonal $\Omega$, the following simple lemma gives a weight function satisfying Condition 1.1 of \cite{Emanuilov1995}. 

\begin{lemma}
\label{lem_dist}
Suppose that $\Omega \subset \R^n$ is open and convex set and that $\p \Omega$ is piecewise smooth. 
Let $x_0 \in \Omega$ and $\epsilon> 0$ satisfy $\overline{B(x_0, \epsilon)} \subset \Omega$.
Then there is $\rho \in C^\infty(\Omega)$
satisfying $\rho < 0$ in $\Omega$,
$|\nabla \rho| = 1$ in $\Omega \setminus B(x_0, \epsilon)$
and $\p_\nu \rho \le 0$ on $\p \Omega$.  
\end{lemma}
\begin{proof}
We choose a strictly negative $\rho \in C^\infty(\Omega)$ such that $\rho(x) = -d(x,x_0)$ outside $B(x_0,\epsilon)$.
Then $\nabla \rho = -(x-x_0)/|x-x_0|$ outside $B(x_0,\epsilon)$.
By convexity, $(x - x_0) \cdot \nu \ge 0$ on $\p \Omega$,
and therefore $\p_\nu \rho \le 0$ on the boundary. \end{proof}

Now \cite[Lemma 1.3]{Emanuilov1995} can be written in the following form:

\begin{lemma}[Global Carleman estimate by Imanuvilov]
\label{th_carlaman_est}
Let $\Omega \subset \R^n$, $x_0 \in \Omega$ and $\epsilon > 0$ 
be as in Lemma \ref{lem_dist}.
Let $\rho$ be the function given by Lemma \ref{lem_dist}.
Let $T > 0$ and 
define $\ell = \lambda \hbar^{-1} (e^{\alpha \rho} - 1)$,
where $\hbar = t(T-t)$, $\alpha, \lambda > 0$.
Then there are $C, \alpha, \lambda_0 > 0$ such that for all 
$w \in C^2([0,T] \times \overline \Omega)$ satisfying 
$w = 0$ on $(0,T) \times \p \Omega$ and all 
$\lambda > \lambda_0$ it holds that 
\begin{align*}
&\int_0^{T} \int_{\Omega} 
\left( \tau |\nabla w|^2 
+ \tau^3 |w|^2 \right) e^{2 \ell} dxdt
\\&\quad\le
C \int_0^{T} \int_\Omega |L w|^2 e^{2\ell} dxdt + 
C \int_0^{T} \int_{B(x_0,\epsilon)} \tau^3 |w|^2 e^{2 \ell} dxdt.
\end{align*}
where $\tau = \lambda \hbar^{-1}$.
\end{lemma}

\begin{proof}[Proof of Theorem \ref{th_cont_stable}]
We write $\mathcal C = C^\infty(0,T; C_0^\infty(\Omega))$
and denote by $\H$ the energy space (\ref{energy_space}).
Recall that $\H$ is continuously embedded in $C(0,T;L^2(\Omega))$, see e.g. \cite[Lemma 11.4]{Renardy2004},
and that $\mathcal C$ is dense in $\H$.
By density it is enough consider the case $u \in \mathcal C$.
Recall also the energy estimate for $v \in \H$, see e.g. \cite[Theorem 11.3]{Renardy2004},
\begin{align}
\label{energy_est_heat}
\norm{v}_{\H} \le C \norm{v(0)}_{L^2(\Omega)} + C \norm{L v}_{L^2(0,T;H^{-1}(\Omega))}.
\end{align}

We write $f = \p_t - \Delta u$, let $v \in \H$ be the solution of
$$
\p_t v - \Delta v = f, \quad v|_{t=0} = 0,
$$
and define $w = u - v$. Then $w$ satisfies 
$\p_t w - \Delta w = 0$. 
Choose $x_0 \in \omega$ and $\epsilon > 0$ such that $B(x_0, \epsilon) \subset \omega$, and apply Lemma \ref{th_carlaman_est} on $w$.
Note that $\tau^3 e^{2 \ell} \to 0$ as $t \to 0$ or $t \to T$.
Let $0 < \delta_0 < \delta$, then $e^{2 \ell}$ is strictly positive on $[\delta_0, T-\delta_0] \times \Omega$.
In particular, 
$$
\norm{w}_{L^2((\delta_0,T-\delta_0) \times \Omega)}
\le C \norm{w}_{L^2((0,T) \times \omega)}.
$$
The estimate (\ref{energy_est_heat}) implies that 
\begin{align*}
\norm{w(\delta)}_{L^2(\Omega)}^2 \le C \norm{w(s)}_{L^2(\Omega)}^2, \quad s \in (0,\delta).
\end{align*}
We integrate this over the interval $(\delta_0,\delta)$  to obtain
\begin{align*}
(\delta-\delta_0) \norm{w(\delta)}_{L^2(\Omega)}^2 \le C \norm{w}_{L^2((\delta_0,T-\delta_0) \times \Omega)}^2.
\end{align*}
By using (\ref{energy_est_heat}) again, we have
$$
\norm{v(\delta)}_{L^2(\Omega)} + \norm{v}_{L^2((0,T) \times \Omega)} \le C \norm{f}_{L^2(0,T; H^{-1}(\Omega))}.
$$
Hence 
\begin{align*}
\norm{u(\delta)}_{L^2(\Omega)} 
&\le C \norm{w}_{L^2((0,T) \times \omega)} + C \norm{f}_{L^2(0,T; H^{-1}(\Omega))}
\\&\le C \norm{u}_{L^2((0,T) \times \omega)} + C \norm{f}_{L^2(0,T; H^{-1}(\Omega))}.
\end{align*}
The claim follows by applying (\ref{energy_est_heat}) once more.
\end{proof}

\vspace{0.5cm}
\noindent{\bf Acknowledgements.}
The authors thank J. Le Rousseau and S. Ervedoza for useful discussions. 
LO was partly supported by EPSRC grant EP/L026473/1. EB was partly
supported by EPSRC grant EP/J002313/2.

\bibliographystyle{abbrv}
\bibliography{main}

\ifdraft{
\listoftodos
}{}

\end{document}